\documentclass[12pt,english]{amsart}

\input xy
\xyoption{all}

\usepackage{comment}
\usepackage[latin1]{inputenc}
\usepackage{amsfonts,amsmath,amssymb,amsthm}
\usepackage{stmaryrd}

\newcommand{\Zz}{\mathbb{Z}}
\newcommand{\Qq}{\mathbb{Q}}

{\theoremstyle{plain}
\newtheorem{theorem}{Theorem}[section]    

\newtheorem{twisting lemma}[theorem]{Twisting lemma}
\newtheorem{lemma}[theorem]{Lemma}       
\newtheorem{proposition}[theorem]{Proposition}  
\newtheorem{corollary}[theorem]{Corollary}   

}
{\theoremstyle{remark}
\newtheorem{definition}[theorem]{Definition}      

}

\def\cm{\hbox{\hbox{\rm C}\kern-5pt{\raise 1pt\hbox{$|$}}}}

\def\lhfl#1#2{\smash{\mathop{\hbox to 12mm{\leftarrowfill}}
\limits^{#1}_{#2}}}
\def\rhfl#1#2{\smash{\mathop{\hbox to 12mm{\rightarrowfill}}
\limits^{#1}_{#2}}}

\def\build#1_#2^#3{\mathrel{
\mathop{\kern 0pt#1}\limits_{#2}^{#3}}}

\def\htrait#1#2{\smash{\mathop{\hbox to 12mm{\hrulefill}}
\limits^{#1}_{#2}}}

\def\sxbullet{{\raise 2pt\hbox{\bf .}}}
\numberwithin{equation}{section}

\begin{document}

\title{Twists of superelliptic curves without rational points}

\author{Fran\c cois Legrand}

\email{flegrand@post.tau.ac.il}

\address{School of Mathematical Sciences, Tel Aviv University, Ramat Aviv, Tel Aviv 6997801, Israel}

\address{Department of Mathematics and Computer Science, the Open University of Israel, Ra'anana 4353701, Israel}

\date{\today}

\maketitle

\begin{abstract}
Let $n\geq 2$ be an integer, $F$ a number field, $O_F$ the integral closure of $\Zz$ in $F$ and $N$ a positive multiple of $n$. The paper deals with degree $N$ polynomials $P(T) \in O_F[T]$ such that the superelliptic curve $Y^n=P(T)$ has twists $Y^n=d\cdot P(T)$ without $F$-rational points. We show that this condition holds if the Galois group of $P(T)$ over $F$ has an element which fixes no root of $P(T)$. Two applications are given. Firstly, we prove that the proportion of degree $N$  polynomials $P(T) \in O_F[T]$ with height bounded by $H$ and such that the associated curve satisfies the desired condition tends to 1 as $H$ tends to $\infty$. Secondly, we connect the problem with the recent notion of non-parametric extensions and give new examples of such extensions with cyclic Galois groups.
\end{abstract}

\section{Introduction}

\subsection{Topic of the paper} Let $n\geq 2$ be an integer, $F$ a number field, $O_F$ the integral closure of $\Zz$ in $F$ and $N$ a positive multiple of $n$. 
Let $$P(T) = a_0 + a_1 T + \cdots + a_{N-1} T^{N-1} + a_NT^N$$
be a polynomial with degree $N$, coefficients in $O_F$ and whose roots all have multiplicity at most $n-1$. In the sequel, the curve
$$C_{P(T)}:Y^n = P(T,Z)=a_0 Z^N + a_1 Z^{N-1} T + \cdots + a_{N-1} Z T^{N-1} + a_NT^N$$ in the weighted projective space $\mathbb{P}_{N/n,1,1}$ is called {\it{the superelliptic curve associated with $P(T)$}}. Given $d \in O_F \setminus \{0\}$, the {\it{d-th twist of $C_{P(T)}$}} is the superelliptic curve $$C_{d \cdot P(T)}:Y^n = d \cdot P(T,Z)$$ associated with the polynomial $d \cdot P(T)$. We refer to \S2.1 for basic terminology.

The purpose of the present paper consists in giving sufficient conditions on the polynomial $P(T)$ for the following condition to hold:

\vspace{1mm}

\noindent
{\rm{($*$)}} {\it{there exists $d \in O_F \setminus \{0\}$, not a $n$-th power in $O_F$, such that $C_{d \cdot P(T)}$ has no $F$-rational point.}}

\vspace{1mm}

\noindent
Note that a necessary condition for condition ($*$) to hold is that the polynomial $P(T)$ has no root in $F$.

\subsection{State-of-the-art}

To our knowledge, this problem has been essentially studied in the case $F=\Qq$ and $n=2$ (and then the degree $N$ of $P(T)$ is even), {\it{i.e.,}} for {\it{hyperelliptic curves}} over $\Qq$. Even in this case, which seems to be the easiest one, the problem is far from being solved. If $N=2$ ({\it{i.e.,}} if the hyperelliptic curve $C_{P(T)}$ has genus $g=0$), then condition {\rm{($*$)}} holds if and only if $P(T)$ has no root in $\Qq$. But the case $N \geq 4$ ({\it{i.e.,}} $g \geq 1$) remains rather mysterious. While examples of hyperelliptic curves with genus $\geq 1$ that have twists with no $\Qq$-rational point can easily be given (take for example $P(T)=T^{N}+1$ ($N \geq 4$) and $d<0$), it seems difficult to decide whether such a given hyperelliptic curve does or not. However, the answer is positive if $N \geq 4$ and $P(T)$ is irreducible \cite[Corollary 4.2]{Sad14}. Moreover, under some conjectures and if $N \geq 6$, the answer is positive if and only if $P(T)$ has no rational root \cite[Corollary 1 and Conjecture 1]{Gra07}. See also \cite[\S11]{Gra07} for similar conjectural conclusions in the case $n \geq 3$ and $F=\Qq$.

\subsection{Main result}

Here we prove the following general result which seems to be unavailable in the literature\footnote{Although this is not explicitly stated there, Theorem \ref{intro 1} may be obtained from \cite[\S4]{Sad14} in the specific case $n=2$ and $F=\Qq$.}.

\begin{theorem} \label{intro 1}
Denote the roots of $P(T)$ by $t_1,\dots,t_N$. Condition {\rm{($*$)}} holds if the Galois group of $P(T)$ over $F$ has an element fixing no root of $P(T)$, {\it{i.e.,}} if \, $\bigcup_{j=1}^N {\rm{Gal}}(F(t_1,\dots,t_N)/F(t_j)) \not= {\rm{Gal}}(F(t_1,\dots,t_N)/F)$.
\end{theorem}

\noindent
See Theorem \ref{thm main} for a more precise conclusion on the set of all elements $d$ such that $C_{d \cdot P(T)}$ has no $F$-rational point. For instance, we obtain that there exists a positive density set $\mathcal{S}$ of non-zero prime ideals $\mathcal{P}$ of $O_F$ such that one can take $d$ equal to any element of $O_F$ that has $\mathcal{P}$-adic valuation 1 for some $\mathcal{P}$ in $\mathcal{S}$. Moreover, we recast results from \cite{Leg15} and \cite{Leg16a} to give explicit examples of polynomials such that the assumption of Theorem \ref{intro 1} holds (for example, it holds if $P(T)$ is irreducible over $F$). Propositions \ref{app 1} and \ref{app 2} give our precise results.

\subsection{Applications}

Two applications of Theorem \ref{intro 1} are given.

\subsubsection{A quantitative version of Theorem \ref{intro 1}}

In the case where $P(T)$ is separable, the assumption of Theorem \ref{intro 1} holds if the Galois group of $P(T)$ over $F$ is isomorphic to the symmetric group $S_N$ (when acting on the roots of $P(T)$). By the Hilbert irreducibility theorem, this happens almost always in a sense. More precisely, combining quantitative versions of the Hilbert irreducibility theorem \cite{Coh81b} and Theorem \ref{intro 1} provides the following result, where we take $n=2$ and $F=\Qq$ for simplicity. See Corollary \ref{thm main 2} for our precise result.

In Corollary \ref{intro 2} below, recall that the {\it{height}} of a polynomial $P(T) \in \Zz[T]$ is the maximum of the absolute values of the coefficients of $P(T)$. 

\begin{corollary} \label{intro 2}
Assume that $n=2$ and let $N\geq1$ be an even integer. Then the proportion $f(H)$ of all degree $N$ separable polynomials $P(T) \in \Zz[T]$ with height at most $H$ and such that the associated curve satisfies condition {\rm{($*$)}} tends to 1 as $H$ tends to $\infty$. More precisely, one has
$$f(H)= 1 - O \bigg(\frac{\log(H)}{\sqrt{H}} \bigg), \quad H \rightarrow \infty.$$
\end{corollary}

\noindent
Corollary \ref{intro 2} may be compared with a result of Bhargava \cite{Bha13} asserting that the proportion of even degree hyperelliptic curves over $\Qq$ with bounded height, given genus and with no $\Qq$-rational point tends to 1 as the genus tends to $\infty$. In our eyes, this result does not provide Corollary \ref{intro 2} as the height of $d \cdot P(T)$ tends to $\infty$ as $|d|$ tends to $\infty$.

\subsubsection{Connections with non-parametric extensions}

Given an indeterminate $T$ and a number field $F$, recall that the {\it{Regular Inverse Galois Problem}} asks whether each finite group $G$ is {\it{a regular Galois group over $F$}}, {\it{i.e.,}} whether $G$ is the Galois group of a Galois extension $E/F(T)$ with $E/F$ {\it{regular}} ({\it{i.e.,}} $E \cap \overline{\Qq} = F$). If there exists such an extension $E/F(T)$ (say for short that $E/F(T)$ is a ``$F$-regular Galois extension with Galois group $G$"), then, by the Hilbert irreducibility theorem, at least one {\it{specialization}} of $E/F(T)$ has Galois group $G$. In particular, $G$ is a Galois group over $F$. Many finite groups have been realized over number fields by this method; see {\it{e.g.}} \cite{MM99}.

In \cite[\S4]{Leg13a}, \cite{Leg15} and \cite{Leg16a}, we consider the following strong variant. We say that a $F$-regular Galois extension $E/F(T)$ with Galois group $G$ is {\it{$G$-parametric}} if every Galois extension of $F$ with Galois group $G$ is a specialization of $E/F(T)$. Although we may expect the anwser to be negative almost always, deciding whether a given $F$-regular Galois extension of $F(T)$ with Galois group $G$ is $G$-parametric is a difficult problem in general (even with $G=\Zz/2\Zz$). In particular, no finite group $G$ such that each $F$-regular Galois extension of $F(T)$ with Galois group $G$ is not $G$-parametric is known.

In \cite{Leg16a}, it is proved that there exists at least one $F$-regular Galois extension of $F(T)$ with Galois group $G$ which is not $G$-parametric as soon as $G$ is a non-trivial regular Galois group over $F$. Moreover, explicit examples are given in \cite[\S4]{Leg13a} and \cite{Leg15} for some specific $F$ and $G$. But methods developed in these papers are rather involved.

Here we combine Theorem \ref{intro 1} and basic tools to give explicit examples of non-parametric extensions with cyclic Galois groups over number fields containing enough roots of unity.

Let $n \geq 2$ be an integer, $F$ a number field, $O_F$ the integral closure of $\Zz$ in $F$ and $P(T) \in O_F[T]$ a non-constant polynomial with degree $N$. Denote the roots of $P(T)$ by $t_1,\dots,t_N$ and set $E=F(T)(\sqrt[n]{P(T)}).$

\begin{corollary} \label{intro 4}
Assume that $n$ divides $N$, $F$ contains the $n$-th roots of unity, $\bigcup_{j=1}^N {\rm{Gal}}(F(t_1,\dots,t_N)/F(t_j)) \not= {\rm{Gal}}(F(t_1,\dots,t_N)/F)$ and $P(T)$ is separable. Then the $F$-regular Galois extension $E/F(T)$ has Galois group $\Zz/n\Zz$ and it is not $\Zz/n\Zz$-parametric.
\end{corollary}

\noindent
More generally, we prove that, given a non-trivial subgroup $H$ of $\Zz/n\Zz$, there exists a Galois extension of $F$ with Galois group $H$ which is not a specialization of $E/F(T)$. Moreover, we relax the assumption that $P(T)$ is separable (at the cost of making the Galois group of $E/F(T)$ smaller); see Proposition \ref{para}. Furthermore, we give a quantitative version of Corollary \ref{intro 4} in the specific case $n=2$ and $F=\Qq$, in the spirit of Corollary \ref{intro 2}. Proposition \ref{para 3} gives our precise result.

\vspace{3mm}

{\bf{Acknowledgments.}} The author is grateful to the anonymous referees for many helpful comments. This work is partially supported by the Israel Science Foundation (grants No. 40/14 and No. 696/13).

\section{Basics}

Below we survey some standard machinery that will be used in the sequel. \S2.1 is devoted to twists of superelliptic curves while \S2.2 deals with prime divisors of polynomials. In \S2.3, we recall some material on heights of polynomials, taken from \cite[\S2]{Coh81b}. Finally, we review some background on function field extensions in \S2.4.

Let $F$ be a number field and $O_F$ the integral closure of $\Zz$ in $F$.

\subsection{Superelliptic curves and twists}

Let $n$ and $N$ be two positive integers such that $n \geq 2$ and $n$ divides $N$.

Consider the equivalence relation $\sim$ on $\overline{\Qq}^3 \setminus \{(0,0,0)\}$ defined as follows:
$$(y_1,t_1,z_1) \sim (y_2,t_2,z_2)$$
if and only if there exists some $\lambda \in \overline{\Qq} \setminus \{0\}$ such that $$(y_2,t_2,z_2) = (\lambda^{N/n}y_1,\lambda t_1, \lambda z_1).$$ The quotient space $(\overline{\Qq}^3 \setminus \{(0,0,0)\} )/ \sim$ is a weighted projective space that is denoted by $$\mathbb{P}_{N/n,1,1}(\overline{\Qq}).$$ Given $(y,t,z) \in \overline{\Qq}^3 \setminus \{(0,0,0)\}$, the corresponding point in $\mathbb{P}_{N/n,1,1}(\overline{\Qq})$ is denoted by $$[y:t:z].$$

Let $P(T) \in O_F[T]$ be a degree $N$ polynomial whose roots all have multiplicity at most $n-1$. Set
$$P(T)=a_0 + a_1 T + \cdots + a_{N-1} T^{N-1} + a_N T^N$$
and
$$P(T,Z)= a_0 Z^N + a_1 Z^{N-1} T + \cdots + a_{N-1} Z T^{N-1} +a_N T^N.$$
The equation $$Y^n= P(T,Z)$$ in $\mathbb{P}_{N/n,1,1}(\overline{\Qq})$ is {\it{the superelliptic curve associated with $P(T)$}}; we denote it by $C_{P(T)}$. 

Let $d$ be in $O_F \setminus \{0\}$. The superelliptic curve $$Y^n = d \cdot P(T,Z)$$ associated with the polynomial $d \cdot P(T)$ is called {\it{the $d$-th twist of $C_{P(T)}$}}; we denote it by $C_{d \cdot P(T)}$. The set of all $F$-rational points on $C_{d \cdot P(T)}$, {\it{i.e.,}} the set of all elements $[y:t:z] \in \mathbb{P}_{N/n,1,1}(\overline{\Qq})$ such that $(y,t,z) \in F^3 \setminus \{(0,0,0)\}$ and $y^n=d \cdot P(t,z)$, is denoted by $$C_{d \cdot P(T)}(F).$$

\subsection{Prime divisors of polynomials}

Given a non-zero prime ideal $\mathcal{P}$ of $O_F$, denote the associated valuation over $F$ by $v_{\mathcal{P}}$. Let $P(T) \in F[T]$ be a non-constant polynomial with degree $N$.

\begin{definition} \label{divisor}
We say that a non-zero prime ideal $\mathcal{P}$ of $O_F$ is {\it{a prime divisor of $P(T)$}} if there exists some $t_0 \in F$ such that $v_\mathcal{P}(P(t_0))>0$.
\end{definition}

\noindent
Assume that $v_\mathcal{P}(a) \geq 0$ for each coefficient $a$ of $P(T)$. If the reduction of $P(T)$ modulo $\mathcal{P}$ has a root in the residue field $O_F/\mathcal{P}$, then $\mathcal{P}$ is a prime divisor of $P(T)$. The converse holds if the leading coefficient $a_N$ of $P(T)$ satisfies $v_\mathcal{P}(a_N)=0$ (in particular, if $P(T)$ is monic).

\vspace{2mm}

The following lemma will be used on several occasions in the sequel.

\begin{lemma} \label{Tchebotarev}
Denote the roots of $P(T)$ by $t_1,\dots,t_N$ and the splitting field of $P(T)$ over $F$ by $L$. Then the following conditions are equivalent:

\vspace{0.5mm}

\noindent
{\rm{(1)}} $\bigcup_{j=1}^N {\rm{Gal}}(L/F(t_j)) \not= {\rm{Gal}}(L/F)$,

\vspace{0.5mm}

\noindent
{\rm{(2)}} there exists a set $\mathcal{S}$ of non-zero prime ideals of $O_F$ that has positive density and such that no prime ideal in $\mathcal{S}$ is a prime divisor of $P(T)$,

\vspace{0.5mm}

\noindent
{\rm{(3)}} there exist infinitely many non-zero prime ideals of $O_F$ each of which is not a prime divisor of $P(T)$.
\end{lemma}

\begin{proof}
Up to replacing $P(T)$ by the product of its distinct irreducible factors (over $F$), we may assume that $P(T)$ is separable.

First, assume that (1) holds. Pick an element $\sigma \in {\rm{Gal}}(L/F) \setminus \bigcup_{j=1}^N {\rm{Gal}}(L/F(t_j))$. By the Tchebotarev density theorem, there exists a positive density set $\mathcal{S}$ of non-zero prime ideals $\mathcal{P}$ of $O_F$  such that the Frobenius associated with $\mathcal{P}$ in $L/F$ is conjugate in ${\rm{Gal}}(L/F)$ to $\sigma$. As $\sigma$ fixes no root of $P(T)$, such a prime $\mathcal{P}$ (except maybe for finitely many exceptions) is not a prime divisor of $P(T)$, as needed for (2).

As the implication (2) $\Rightarrow$ (3) is obvious, it remains to prove the implication (3) $\Rightarrow$ (1). Assume that (1) fails. Let $\mathcal{P}$ be a non-zero prime ideal of $O_F$ that does not ramify in $L/F$ and $\sigma$ the associated Frobenius in this extension. As (1) has been assumed to fail, $\sigma$ fixes a root of $P(T)$. In particular, $\mathcal{P}$ is a prime divisor of $P(T)$ (except maybe for finitely many exceptions). Hence (3) does not hold either.
\end{proof}

\subsection{Heights of polynomials}

Denote the degree of the extension $F/\Qq$ by $\delta$. Let $\omega_1, \dots, \omega_\delta$ be a $\Zz$-basis of $O_F$. Define 

\noindent
- the {\it{height of $a \in O_F$}} as ${\rm{max}} (|a_1|^\delta, \dots, |a_\delta|^\delta)$
where $a$ is uniquely written as $$a=a_1 \cdot \omega_1 + \cdots + a_\delta \cdot \omega_\delta,$$

\noindent
- the {\it{height of $P(T) \in O_F[T]$}} as the maximum of the heights of the coefficients of $P(T)$.

\noindent
Note that these definitions depend on the choice of a $\Zz$-basis of $O_F$.
\subsection{Function field extensions}

Given an indeterminate $T$, let $E/F(T)$ be a finite Galois extension with Galois group $G$ and such that $E/F$ is {\it{regular}} ({\it{i.e.,}} $E \cap \overline{\Qq}=F$). For short, we say that $E/F(T)$ is a $F$-regular Galois extension with Galois group $G$.

Recall that a point $t_0 \in \mathbb{P}^1(\overline{\Qq})$ is {\it{a branch point of $E/F(T)$}} if the prime ideal $(T-t_0)  \overline{\Qq}[T-t_0]$ \footnote{Replace $T-t_0$ by $1/T$ if $t_0 = \infty$.} ramifies in the integral closure of $\overline{\Qq}[T-t_0]$ in the {\it{compositum}} of $E$ and $\overline{\Qq}(T)$ (in a fixed algebraic closure of $F(T)$). The extension $E/F(T)$ has only finitely many branch points.

Given a point $t_0 \in \mathbb{P}^1(F)$, not a branch point, the residue extension of $E/F(T)$ at a prime ideal $\mathcal{P}$ lying over $(T-t_0) {F}[T-t_0]$ is denoted by ${E}_{t_0}/F$ and called {\it{the specialization of ${E}/F(T)$ at $t_0$}}. This does not depend on the choice of the prime $\mathcal{P}$  lying over $(T-t_0) {F}[T-t_0]$ as ${E}/F(T)$ is Galois. The extension $E_{t_0}/F$ is Galois with Galois group a subgroup of $G$, namely the decomposition group of ${E}/F(T)$ at $\mathcal{P}$.

Given a subgroup $H$ of $G$, the extension $E/F(T)$ is said to be {\it{$H$-parametric}} if each Galois extension of $F$ with Galois group $H$ is a specialization of $E/F(T)$.

\section{Proof of Theorem \ref{intro 1}}

This section is organized as follows. In \S3.1, we state Theorem \ref{thm main} which is a more precise version of Theorem \ref{intro 1}. Theorem \ref{thm main} is proved in \S3.2. Next, we discuss the converse of Theorem \ref{thm main} in \S3.3. Finally, \S3.4 is devoted to Propositions \ref{app 1} and \ref{app 2} which give explicit examples of polynomials satisfying the assumption of Theorem \ref{thm main}.

\subsection{Statement of Theorem \ref{thm main}}

Let $n \geq 2$ be an integer, $F$ a number field, $O_F$ the integral closure of $\Zz$ in $F$ and $N$ a positive multiple of $n$. Let $P(T) = a_0 + a_1 T + \cdots +a_{N-1} T^{N-1} + a_N T^N$
be a polynomial with degree $N$, coefficients in $O_F$ and whose roots all have multiplicity at most $n-1$. Denote the roots of $P(T)$ by $t_1,\dots,t_N$ and the splitting field of $P(T)$ over $F$ by $L$. Given a non-zero prime ideal $\mathcal{P}$ of $O_F$, denote the associated valuation over $F$ by $v_\mathcal{P}$.

\begin{theorem} \label{thm main}
Assume that the following condition holds:

\vspace{0.5mm}

\noindent
{\rm{(H)}} $\bigcup_{j=1}^N {\rm{Gal}}(L/F(t_j)) \not= {\rm{Gal}}(L/F)$. 

\vspace{0.5mm}

\noindent
Then there exists a set $\mathcal{S}$ of non-zero prime ideals of $O_F$ that has positive density and that satisfies the following property.

\vspace{0.5mm}

\noindent
{\rm{($**$)}} One has $C_{d \cdot P(T)}(F) = \emptyset$ for each element $d \in O_F \setminus \{0\}$ for which there exists a prime $\mathcal{P}$ in $\mathcal{S}$ such that $v_{\mathcal{P}}(d) > 0$ and $n {\not \vert} v_{\mathcal{P}}(d)$.
\end{theorem}

\subsection{Proof of Theorem \ref{thm main}}

By condition (H) and the implication (1) $\Rightarrow$ (2) of Lemma \ref{Tchebotarev}, there exists a positive density set $\mathcal{S}$ of non-zero prime ideals of $O_F$ each of which is not a prime divisor of $P(T)$. Moreover, the polynomial $P(T)$ has no root in $F$. In particular, one has $a_0 \not=0$. Up to dropping finitely many primes, we may assume that each prime $\mathcal{P}$ in $\mathcal{S}$ satisfies $v_\mathcal{P}(a_0) = 0$ and $v_\mathcal{P}(a_N)=0$.

Let $d$ be in $\mathcal{O}_F \setminus \{0\}$. Assume that there exists a prime ideal $\mathcal{P} \in \mathcal{S}$ such that $v_\mathcal{P}(d) >0$ and $n {\not \vert} v_\mathcal{P}(d)$. Fix such a prime ideal $\mathcal{P}$.

Suppose $C_{d \cdot P(T)}$ has a $F$-rational point $[y:t:z]$. If $z=0$, one has
\begin{equation} \label{eq-1}
y^n = d \cdot a_N t^N.
\end{equation}
In particular, one has $y \not=0$ and $t \not=0$. By the definition of $\mathcal{S}$ and \eqref{eq-1}, one has $n \cdot v_\mathcal{P}(y) = v_\mathcal{P}(d) + N \cdot v_\mathcal{P}(t)$. As $n$ divides $N$, we get that $n$ divides $v_\mathcal{P}(d)$, which cannot happen. Hence $z \not=0$. As already said, $P(T)$ has no root in $F$. Hence $y \not=0$. Set $$y'=y / z^{N/n} \, \, {\rm{and}} \, \, t'=t/z.$$ Note that $y'$ and $t'$ are uniquely determined, $y' \not=0$ and $t'$ is not a root of $P(T)$. Moreover, one has 
\begin{equation} \label{eq 0}
y'^n = d \cdot P(t').
\end{equation}
 If $v_\mathcal{P}(P(t')) = 0$, \eqref{eq 0} provides $n \cdot v_\mathcal{P}(y') = v_\mathcal{P}(d)$. Then $n \vert v_\mathcal{P}(d)$, which cannot happen. Hence 
\begin{equation} \label{eq -1-}
v_\mathcal{P}(P(t')) \not=0.
\end{equation}
If $t'=0$, \eqref{eq 0} gives $y'^n = d \cdot a_0$. By the definition of $\mathcal{S}$, we get $n \vert v_\mathcal{P}(d)$, which cannot happen. Hence $t' \not=0$. If $v_\mathcal{P}(t') \geq 0$, we get $v_\mathcal{P}(P(t')) \geq 0$. By \eqref{eq -1-}, one then has $v_\mathcal{P}(P(t')) >0$. Hence $\mathcal{P}$ is a prime divisor of $P(T)$, which contradicts the definition of $\mathcal{S}$. Then $v_\mathcal{P}(t')<0$. By the definition of $\mathcal{S}$, we get 
\begin{equation} \label{eq 2}
v_\mathcal{P}(P(t')) = v_\mathcal{P}(t'^N) = N \cdot v_\mathcal{P}(t').
\end{equation}
Combining \eqref{eq 0} and \eqref{eq 2} then provides
$n \cdot v_\mathcal{P}(y') = v_\mathcal{P}(d) +N \cdot v_\mathcal{P}(t').$
As $n$ divides $N$, we get that $n$ divides $v_\mathcal{P}(d)$, which cannot happen. Hence $C_{d \cdot P(T)}(F) = \emptyset$, thus ending the proof of Theorem \ref{thm main}.

\subsection{On the converse of Theorem \ref{thm main}}

Let $\mathcal{S}^*$ be the subset of the set of all non-zero prime ideals of $O_F$ that is maximal with respect to condition ($**$) of Theorem \ref{thm main}. The argument given in \S3.2 shows that $\mathcal{S}^*$ contains all but finitely many non-zero prime ideals of $O_F$ that are not prime divisors of $P(T)$ (either under condition (H) or not).

Below we provide some converse to this last conclusion.

\begin{proposition} \label{converse}
Assume that $P(T)$ is separable. Then $\mathcal{S}^*$ and the set of all non-zero prime ideals of $O_F$ that are not prime divisors of $P(T)$ differ only in finitely many prime ideals.
\end{proposition}

\begin{proof}
Let $\mathcal{P}$ be a prime divisor of $P(T)$. Assume that $v_\mathcal{P}(a_N)=0$ and $\mathcal{P}$ does not contain the discriminant $\Delta \not=0$ of $P(T)$. Then there exists $t \in O_F$ such that $P(t) \not=0$ and $v_\mathcal{P}(P(t))$ is a positive integer. We claim that we may assume $v_\mathcal{P}(P(t))=1$. Indeed, assume that
\begin{equation} \label{3.2}
v_\mathcal{P}(P(t)) \geq 2.
\end{equation}
Let $x_\mathcal{P} \in O_F$ be a generator of the maximal ideal $\mathcal{P} ({O_F})_\mathcal{P}$ of the localization $({O_F})_\mathcal{P}$ of $O_F$ at $\mathcal{P}$. By the Taylor formula, one has
\begin{equation} \label{3.3}
P(t+x_\mathcal{P}) = P(t) + x_\mathcal{P} P'(t) + x_\mathcal{P}^2 R_\mathcal{P}
\end{equation}
for some $R_\mathcal{P} \in ({O_F})_\mathcal{P}$. As $\mathcal{P}$ does not contain $\Delta$, one has 
\begin{equation} \label{3.4}
v_\mathcal{P}(P'(t))=0.
\end{equation} 
Combining \eqref{3.2}, \eqref{3.3} and \eqref{3.4} then provides $v_\mathcal{P}(P(t + x_\mathcal{P})) =1$, thus proving our claim (up to replacing $t$ by $t+x_\mathcal{P}$). From now on, we assume that $v_\mathcal{P}(P(t))=1$. Now, set $d=P(t)^{n-1} \in O_F \setminus \{0\}$. Then $v_\mathcal{P}(d) >0$, $n {\not \vert} v_\mathcal{P}(d)$ and the $F$-rational point $[P(t):t:1]$ lies on the $d$-th twist $C_{d \cdot P(T)}$ of $C_{P(T)}$. Hence $\mathcal{P}$ is not in $\mathcal{S}^*$, as needed.
\end{proof}

We then get the following further conclusions.

\vspace{2mm}

\noindent
(1) Assume that condition (H) holds. Then, by the Tchebotarev density theorem and \S3.2, the set $\mathcal{S}^*$ contains a set of non-zero prime ideals of $O_F$ that has positive density $\delta$ equal to 
$$1 - \frac{|\bigcup_{j=1}^N {\rm{Gal}}(L/F(t_j))|}{|{\rm{Gal}}(L/F)|}.$$
Moreover, if $P(T)$ is separable, then $\mathcal{S}^*$ has density $\delta$ (Proposition \ref{converse}).

\vspace{2mm}

\noindent
(2) Assume that $P(T)$ is separable and condition (H) fails. Then, by Proposition \ref{converse} and the implication (3) $\Rightarrow$ (1) of Lemma \ref{Tchebotarev}, the set $\mathcal{S}^*$ is finite. We then get the following converse to Theorem \ref{thm main}.

\begin{corollary} \label{converse 2}
Assume that there exists an infinite set $\mathcal{S}$ of non-zero prime ideals of $O_F$ that satisfies condition {\rm{($**$)}} of Theorem \ref{thm main}. Then either condition {\rm{(H)}} holds or the polynomial $P(T)$ is not separable.
\end{corollary}

\subsection{Explicit examples}

In Propositions \ref{app 1} and \ref{app 2} below, we provide explicit examples of polynomials $P(T)$ such that condition (H) holds. For simplicity, we only give examples of such polynomials that are separable. 

Let $F$ be a number field, $O_F$ the integral closure of $\Zz$ in $F$ and $P(T) \in O_F[T]$ a separable polynomial with degree $N$.

\subsubsection{Statement of Proposition \ref{app 1}}
Set
$$P(T)=P_1(T) \cdots P_r(T) \cdot P_{r+1}(T) \cdots P_s(T)$$ 
where $P_1(T), \dots, P_s(T)$ are distinct polynomials that are irreducible over $F$. For each $j \in \{1,\dots,s\}$, let $F_j$ be the field generated over $F$ by one root of $P_j(T)$ and $L_j$ the splitting field of $P_j(T)$ over $F$.

\begin{proposition} \label{app 1}

Assume that the following three conditions hold:

\vspace{0.5mm}

\noindent
{\rm{(1)}} $P(T)$ has no root in $F$,

\vspace{0.5mm}

\noindent
{\rm{(2)}} given $j \in \{r+1,\dots,s\}$, one has $F_j=F_{i_j}$ for some $i_j \in \{1,\dots,r\}$,

\vspace{0.5mm}

\noindent
{\rm{(3)}} the fields $L_1,\dots,L_r$ are linearly disjoint over $F$.

\vspace{0.5mm}

\noindent
Then condition {\rm{(H)}} holds.
\end{proposition}

\noindent
For example, conditions (1), (2) and (3) hold if

\vspace{0.5mm}

\noindent
- $P(T)$ is irreducible over $F$ and $N \geq 2$ (with $r=s=1$) or if

\vspace{0.5mm}

\noindent
- $P(T)$ has no root in $F$ and $N=4$. Indeed, by the previous case, we may assume $s \geq 2$. As $P(T)$ has no root in $F$ and $N=4$, one has $s=2$ and $P_1(T)$, $P_2(T)$ both have degree 2. Hence $F_1=L_1$ and $F_2=L_2$. If $F_1=F_2$, conditions (2) and (3) hold with $r=1$. If $F_1 \not=F_2$, condition (2) is empty with $r=2$ and condition (3) holds as the distinct quadratic fields $L_1$ and $L_2$ are linearly disjoint over $F$.

\begin{proof}
The proof below is similar to part of the proof of \cite[Corollary 6.1]{Leg15} which corresponds to the case $r=s$. We reproduce it below with the necessary adjustments for the higher generality.

Given $j \in \{1,\dots,r\}$, let $G_j$ be the Galois group of $P_j(T)$ over $F$, $N_j$ the degree of $P_j(T)$ and $\nu_j \, : \, G_j \rightarrow S_{N_j}$ the action of $G_j$ on the roots of $P_j(T)$. By condition (3), the Galois group of $P_1(T) \dots P_r(T)$ over $F$ is isomorphic to $G_1 \times \cdots \times G_r$ and $\nu_1 \times \dots \times \nu_r : G_1 \times \dots \times G_r \rightarrow S_{N_1+\cdots+N_r}$ corresponds to its action on the roots of $P_1(T) \dots P_r(T)$. Given $j \in \{1,\dots,r\}$, the irreducible polynomial $P_j(T)$ has degree $N_j \geq 2$ (condition (1)). Then there exists some element $g_j \in G_j$ such that $\nu_j(g_j)$ has no fixed point. Hence $(g_1,\dots,g_r)$ is an element of the Galois group of $P_1(T) \dots P_r(T)$ over $F$ which fixes no root of $P_1(T) \dots P_r(T)$. We may then use the implication (1) $\Rightarrow$ (3) of Lemma \ref{Tchebotarev} to obtain that there exist infinitely many non-zero prime ideals of $O_F$ each of which is not a prime divisor of $P_1(T) \dots P_r(T)$. By condition (2), $P_1(T) \dots P_r(T)$ and $P(T)$ have the same prime divisors (up to finitely many). Hence the polynomial $P(T)$ also satisfies condition (3) of Lemma \ref{Tchebotarev}. It then remains to apply the implication (3) $\Rightarrow$ (1) of Lemma \ref{Tchebotarev} to conclude.
\end{proof}

\subsubsection{Statement of Proposition \ref{app 2}}

The following result rests on Lemma \ref{Tchebotarev} and \cite[Proposition 3.5]{Leg16a}.

\begin{proposition} \label{app 2}
Assume that $P(T)$ is monic and neither 0 nor 1 is a root of $P(T)$. Then there exist infinitely many positive integers $k$ such that the Galois group of $P(T^k)$ over $F$ satisfies condition {\rm{(H)}}.
\end{proposition}

\noindent
Note that both assumptions $P(0) \not=0$ and $P(1) \not=0$ are necessary for the conclusion to hold. Moreover, explicit sufficient conditions on a given positive integer $k$ for the polynomial $P(T^k)$ to satisfy condition (H) can be given; see \cite[Propositions 4.1, 4.2 and 4.3]{Leg16a}.

\section{Proof of Corollary \ref{intro 2}}

The aim of this section is to prove Corollary \ref{thm main 2} below which is a more general version of Corollary \ref{intro 2}. Let $F$ be a number field, $O_F$ the integral closure of $\Zz$ in $F$, $n \geq 2$ an integer and $N \geq 1$ a multiple of $n$.

\subsection{Statement of Corollary \ref{thm main 2}}

Given a real number $H \geq 1$, let $$S(N,H)$$ be the set of all degree $N$ polynomials $P(T)$ with coefficients in $O_F$, height\footnote{We fix beforehand a $\Zz$-basis of $O_F$.} at most $H$ and whose roots all have multiplicity $\leq n-1$. Let 
$$S'(N,H)$$ be the subset of all elements $P(T)$ in $S(N,H)$ such that there exists a set of non-zero prime ideals of $O_F$ that has positive density and that satisfies condition ($**$) of Theorem \ref{thm main}.

\begin{corollary} \label{thm main 2}
One has
$$\frac{|S'(N,H)|}{|S(N,H)|} = 1 - O \bigg(\frac{\log(H)}{\sqrt{H}} \bigg), \quad H \to \infty.$$
\end{corollary}

\subsection{Proof of Corollary \ref{thm main 2}}

Given a positive real number $H$, an element $P(T)$ of $S(N,H)$ is in $S'(N,H)$ if it is separable and its Galois group over $F$, viewed as a permutation group of the roots, contains a $N$-cycle (Theorem \ref{thm main}). In particular, one has
\begin{equation} \label{4.1}
|S(N,H)| - |S'(N,H)| \leq {\rm{Er}}(N,H)
\end{equation}
where $${\rm{Er}}(N,H)$$
denotes the number of degree $N$ polynomials with coefficients in $O_F$, height at most $H$, whose roots all have multiplicity at most $n-1$ and that are either non-separable or that are separable and have Galois group isomorphic to a proper subgroup of $S_N$. Given algebraically independent indeterminates $T_0, \dots, T_N$ and $Y$, the Galois group of $$T_0 + T_1 Y + \dots + T_{N-1} Y^{N-1} + T_NY^N$$ over $F(T_0,\dots,T_N)$ is isomorphic to $S_N$. \cite[Theorem 2.1]{Coh81b} then provides
\begin{equation} \label{4.2}
{\rm{Er}}(N,H) = O(H^{N + (1/2)} \cdot \log(H)) , \quad H \to \infty.
\end{equation}
Next, combine \eqref{4.1} and \eqref{4.2} to get
$$1 - \frac{|S'(N,H)|}{|S(N,H)|} = \frac{O(H^{N+(1/2)} \cdot \log(H))}{|S(N,H)|}.$$
It then remains to use that $$|S(N,H)| \sim 2^{[F:\Qq] \cdot (N+1)} \cdot H^{N+1}$$ as $H$ tends to $\infty$ to complete the proof of Corollary \ref{thm main 2}.

\section{Proof of Corollary \ref{intro 4}}

This section is organized as follows. In \S5.1, we state Proposition \ref{para} which is a more general version of Corollary \ref{intro 4}. Proposition \ref{para} is then proved in \S5.2. Finally, \S5.3 is devoted to Proposition \ref{para 3} which is a quantitative version of Proposition \ref{para} in the case $n=2$ and $F=\Qq$.

Let $n \geq 2$ be an integer, $F$ a number field, $O_F$ the integral closure of $\Zz$ in $F$, $T$ an indeterminate and $P(T) \in O_F[T]$ a non-constant degree $N$ polynomial. Denote the roots of $P(T)$ by $t_1, \dots,t_N$, the leading coefficient by $a_N$ and the splitting field of $P(T)$ over $F$ by $L$. Set
$$P(T)=P_1(T)^{e_1} \cdots P_s(T)^{e_s}$$ where $e_1, \dots, e_s$ are positive integers and $P_1(T), \dots, P_s(T)$ are distinct polynomials that are irreducible over $F$. Set $$e={\rm{gcd}}(n,e_1,\dots,e_s),$$ $$n'=\frac{n}{e}, \quad e'_1=\frac{e_1}{e}, \dots, e'_s=\frac{e_s}{e}.$$
Finally, fix a $n$-th root $\sqrt[n]{P(T)}$ of $P(T)$ and set $E=F(T)(\sqrt[n]{P(T)}).$

\subsection{Statement of Proposition \ref{para}}

Given a non-zero prime ideal $\mathcal{P}$ of $O_F$, denote the associated valuation over $F$ by $v_\mathcal{P}$.

\begin{proposition} \label{para}
Assume that the following conditions hold:

\vspace{0.5mm}

\noindent
{\rm{(hyp-1)}} $n$ divides $N$,

\vspace{0.5mm}

\noindent
{\rm{(hyp-2)}} $e_j \leq n-1$ for each $j \in \{1, \dots,s\}$,

\vspace{0.5mm}

\noindent
{\rm{(hyp-3)}} $F$ contains the $n'$-th roots of unity,

\vspace{0.5mm}

\noindent
{\rm{(hyp-4)}} $\bigcup_{j=1}^N {\rm{Gal}}(L/F(t_j)) \not= {\rm{Gal}}(L/F).$

\vspace{0.5mm}

\vspace{0.5mm}

\noindent
Then the following two conclusions hold.

\vspace{0.5mm}

\noindent
{\rm{(1)}} The extension $E/F(T)$ is $F$-regular, Galois and its Galois group is isomorphic to $\Zz/n'\Zz$.

\vspace{0.5mm}

\noindent
{\rm{(2)}} There exists a set $\mathcal{S}$ of non-zero prime ideals of $O_F$ that has positive density and that satisfies the following property. 

\vspace{0.5mm}

\noindent
{\rm{($***$)}} Let $d$ be in $O_F \setminus \{0\}$. Assume that there exists a prime ideal $\mathcal{P} \in \mathcal{S}$ such that $v_\mathcal{P}(d) >0$ and $n'$ does not divide $v_\mathcal{P}(d)$. Then the extension $F(\sqrt[n']{d})/F$ does not occur as a specialization of $E/F(T)$. 

\vspace{0.5mm}

\noindent
In particular, the extension $E/F(T)$ is not $H$-parametric for each non-trivial subgroup $H$ of $\Zz/n'\Zz$.
\end{proposition}

\subsection{Proof of Proposition \ref{para}}

We break the proof into two parts.

\subsubsection{Proof of Proposition \ref{para} under an extra assumption}

First, we prove the following version of Proposition \ref{para}.

\begin{proposition} \label{para 2}
The two conclusions of Proposition \ref{para} hold under conditions {\rm{(hyp-1)}}, {\rm{(hyp-2)}}, {\rm{(hyp-3)}}, {\rm{(hyp-4)}} and {\rm{(hyp-5)}} with

\vspace{0.5mm}

\noindent
{\rm{(hyp-5)}} $e=1$ ({\it{i.e.,}} $n'=n$).
\end{proposition}

We break the proof into three parts.

\vspace{2mm}

\noindent
5.2.1.1. {\it{Proof of conclusion {\rm{(1)}}.}}
Assume that the polynomial $Y^n - {P}(T)$ is reducible over $\overline{\Qq}(T)$. If 4 does not divide $n$, then, by the Capelli lemma \cite[Chapter VI, \S9, Theorem 9.1]{Lan02}, there exists a prime number $p$ dividing $n$ and a polynomial $R(T)$ in $\overline{\Qq}[T]$ such that 
$${P}(T)= R(T)^p.$$ This implies that $e_1, \dots, e_s$ all are multiples of $p$, which cannot happen by condition (hyp-5). Now, if 4 divides $n$, then, by the Capelli lemma and the previous case, there exists $R(T)$ in $\overline{\Qq}[T]$ such that $${P}(T)=  -4 R(T)^4.$$ Then $e_1, \dots, e_s$ all are even, which contradicts condition (hyp-5). Hence $Y^n-P(T)$ is irreducible over $\overline{\Qq}(T)$, {\it{i.e.,}} $E/F(T)$ has degree $n$ and $E/F$ is regular. As $F$ contains the $n$-th roots of unity (conditions (hyp-3) and (hyp-5)) and $E/F(T)$ has degree $n$, the extension $E/F(T)$ is Galois and has Galois group isomorphic to $\Zz/n\Zz$.

\vspace{2mm}

\noindent
5.2.1.2. {\it{From specializations of $E/F(T)$ to rational points on twists of the superelliptic curve $C_{P(T)}$ and vice-versa.}} First, we determine the set of branch points of the extension $E/F(T)$.

\begin{lemma} \label{lemma 1}
The set of branch points of $E/F(T)$ is $\{t_1,\dots,t_N\}$.
\end{lemma}

\begin{proof}
Let $t \in \overline{\Qq}$ be a branch point of $E/F(T)$. As $F$ contains the $n$-th roots of unity, $E$ is the splitting field over $F(T)$ of $Y^n-P(T)$. Then $Y^n-P(t)$ has a multiple root, {\it{i.e.,}} $t$ is a root of $P(T)$. 

Conversely, let $t \in \overline{\Qq}$ be a root of the polynomial $P(T)$, say of the polynomial $P_1(T)$. Then $t$ is a branch point of $E/F(T)$ if and only if 
$$\sqrt[n]{P_1(T)^{e_1} \cdot P_2(T)^{e_2} \cdots P_s(T)^{e_s}}$$ is not in $\overline{\Qq}((T-t))$. As $t$ is not a root of $P_2(T)^{e_2} \cdots P_s(T)^{e_s}$, $t$ is not a branch point of the extension $$F(T)(\sqrt[n]{P_2(T)^{e_2} \cdots P_s(T)^{e_s}})/F(T),$$ {\it{i.e.,}} $$\sqrt[n]{P_2(T)^{e_2} \cdots P_s(T)^{e_s}}$$ is in $\overline{\Qq}((T-t))$. Hence $t$ is a branch point of $E/F(T)$ if and only if $\sqrt[n]{P_1(T)^{e_1}}$ is not in $\overline{\Qq}((T-t))$.

Assume by contradiction that $\sqrt[n]{P_1(T)^{e_1}}$ is in $\overline{\Qq}((T-t))$. Then there exists $A(T-t) \in \overline{\Qq}[[T-t]]$ such that $$P_1(T)^{e_1}= A(T-t)^n.$$
We then get that the multiplicity of $t$ as a root of the polynomial $P_1(T)^{e_1}$, which is equal to $e_1$, is a multiple of $n$. This cannot happen as $1 \leq e_1 \leq n-1$ (condition (hyp-2)). Hence $\sqrt[n]{P_1(T)^{e_1}}$ is not in $\overline{\Qq}((T-t))$, {\it{i.e.,}} $t$ is a branch point of $E/F(T)$.

It remains to show that $\infty$ is not a branch point of $E/F(T)$. Set
$$P(T)=a_0 + a_1 T + \cdots + a_{N-1}T^{N-1} + a_N T^N.$$
One has
$$P(T)=(T^{N/n})^{n} \Big(\frac{a_0}{T^N}  + \frac{a_1}{T^{N-1}} + \cdots + \frac{a_{N-1}}{T} + a_N \Big).$$
Set $U=1/T$ and $Q(U)= a_N + a_{N-1} U + \cdots + a_1 U^{N-1} + a_0 U^N $. As $N/n$ is an integer (condition (hyp-1)), one has $E=F(U)(\sqrt[n]{Q(U)}).$ Since 0 is not a root of $Q(U)$, we get from the first part of the proof that $\infty$ is not a branch point of $E/F(T)$, thus ending the proof.
\end{proof}

Now, we make the specializations of the extension $E/F(T)$ explicit.

\begin{lemma} \label{lemma 2}
{\rm{(1)}} One has $E_{t_0} = F(\sqrt[n]{P(t_0)})$ for each $t_0 \in F \setminus \{t_1,\dots,t_N\}$.

\vspace{0.5mm}

\noindent
{\rm{(2)}} One has $E_\infty=F(\sqrt[n]{a_N}).$
\end{lemma}

\begin{proof}
Given $t_0 \in F \setminus \{t_1,\dots,t_N\}$, the polynomial $Y^n-P(t_0)$ is separable. As $E$ is the splitting field of $Y^n-P(T)$ over $F(T)$, the field $E_{t_0}$ is the splitting field of $Y^n-P(t_0)$ over $F$. Then, by condition (hyp-3), one has $E_{t_0} = F(\sqrt[n]{P(t_0)})$, as needed for part (1). The proof of part (2) is similar. Consider the polynomial $Q(U)$ from the last part of the proof of Lemma \ref{lemma 1}. As 0 is not a root of $Q(U)$, the polynomial $Y^n-Q(0)$ is separable. Hence the field $E_\infty$ is equal to its splitting field over $F$, {\it{i.e.,}} one has $E_\infty=F(\sqrt[n]{a_N})$, thus ending the proof.
\end{proof}

Finally, we connect specializations of the extension $E/F(T)$ and $F$-rational points on twists of $C_{P(T)}$.

\begin{lemma} \label{lemma 3}
Let $d$ be in $O_F \setminus \{0\}$. The following two conditions are equivalent:

\vspace{0.25mm}

\noindent
{\rm{(1)}} the extension $F(\sqrt[n]{d})/F$ occurs as a specialization of $E/F(T)$,

\vspace{0.5mm}

\noindent
{\rm{(2)}} there exists a positive integer $m$ such that $m < n$, $(m,n)=1$ and $C_{d^m \cdot {P}(T)}$ has a $F$-rational point $[y:t:z]$ with $y \not=0$.
\end{lemma}

\begin{proof}
First, assume that $F(\sqrt[n]{d})/F$ occurs as the specialization of $E/F(T)$ at $t_0$ for some $t_0 \in F\setminus \{t_1,\dots,t_N\}$. We may apply part (1) of Lemma \ref{lemma 2} to get $F(\sqrt[n]{d}) = F(\sqrt[n]{P(t_0)}).$ The Kummer theory then gives $$y^n=d^m \cdot P(t_0)$$
for some $y \in F\setminus \{0\}$ and some integer $m$ such that $1 \leq m < n$ and $(m,n)=1$ (see {\it{e.g.}} \cite[page 18]{Ade01}). Hence the $F$-rational point $[y:t_0:1]$ lies on $C_{d^m \cdot P(T)}$. In the case $F(\sqrt[n]{d})/F$ occurs as the specialization of $E/F(T)$ at $\infty$, a similar argument, combined with part (2) of Lemma \ref{lemma 2}, shows that there exists some $y \in F \setminus \{0\}$ and some integer $m$ such that $1 \leq m < n$, $(m,n)=1$ and the $F$-rational point $[y:1:0]$ lies on $C_{d^m \cdot P(T)}$.

Conversely, assume that there exists an integer $m$ such that $1 \leq m < n$, $(m,n)=1$ and $C_{d^m \cdot P(T)}$ has a $F$-rational point $[y:t:z]$ with $y \not=0$. First, assume that $z=0$. Then $y^n=d^m \cdot a_N t^N$. This provides $F(\sqrt[n]{d^{n-m}}) = F(\sqrt[n]{a_N})$. As $n-m$ is coprime to $n$, one has $F(\sqrt[n]{d^{n-m}})=F(\sqrt[n]{d})$ and, by part (2) of Lemma \ref{lemma 2}, one has $F(\sqrt[n]{a_N})=E_\infty$. Then $F(\sqrt[n]{d})=E_\infty$, as needed. Now, assume that $z \not=0$. We then get 
\begin{equation} \label{3.1}
y^n=d^m \cdot z^N \cdot P(t/z).
\end{equation}
As $y \not=0$, $t/z$ is not a branch point (Lemma \ref{lemma 1}). By \eqref{3.1}, one has $F(\sqrt[n]{d^{n-m}}) = F(\sqrt[n]{P(t/z)})$. But $F(\sqrt[n]{d^{n-m}})=F(\sqrt[n]{d})$ and $F(\sqrt[n]{P(t/z)})$ $=E_{t/z}$ (part (1) of Lemma \ref{lemma 2}). Hence $F(\sqrt[n]{d})= E_{t/z}$, as needed.
\end{proof}

\noindent
5.2.1.3. {\it{Proof of conclusion {\rm{(2)}}.}}  As conditions (hyp-1), (hyp-2) and (hyp-4) hold, we may apply Theorem \ref{thm main}. Then there exists a set $\mathcal{S}$ of non-zero prime ideals of $O_F$ that has positive density and that satisfies condition ($**$), {\it{i.e.,}} one has $C_{d \cdot P(T)}(F) = \emptyset$ for each element $d \in O_F \setminus \{0\}$ for which there exists a prime $\mathcal{P}$ in $\mathcal{S}$ such that $v_{\mathcal{P}}(d) > 0$ and $n {\not \vert} v_{\mathcal{P}}(d)$. Fix such a $d \in O_F \setminus \{0\}$ and a prime $\mathcal{P} \in \mathcal{S}$ such that $v_{\mathcal{P}}(d) > 0$ and $n {\not \vert} v_{\mathcal{P}}(d)$. Let $m$ be a positive integer such that $m < n$ and $(m,n)=1$. Then one has $v_{\mathcal{P}}(d^m) > 0$ and $n {\not \vert} v_{\mathcal{P}}(d^m)$. Hence $C_{d^m \cdot P(T)}(F) = \emptyset$. By the implication (1) $\Rightarrow$ (2) of Lemma \ref{lemma 3}, the extension $F(\sqrt[n]{d})/F$ does not occur as a specialization of $E/F(T)$.

Let $H$ be a non-trivial subgroup of $\Zz/n\Zz$, {\it{i.e.,}} $H=\Zz/n''\Zz$ for some integer $n'' \geq 2$ dividing $n$. It remains to show that one may require a Galois extension $F(\sqrt[n]{d})/F$ as above to have Galois group $H$. Let $d$ be in $O_F \setminus \{0\}$ such that $v_\mathcal{P}(d)=1$ for some prime ideal $\mathcal{P} \in \mathcal{S}$. Set $m={n/n''}$. Then
$$F(\sqrt[n]{d^m}) = F(\sqrt[n'']{d}).$$
Since $v_\mathcal{P}(d)=1$, one may apply the Eisenstein criterion to get that the polynomial $Y^{n''} - d$ is irreducible over $F$, {\it{i.e.,}} $F(\sqrt[n'']{d})/F$ has degree $n''$. Hence $F(\sqrt[n]{d^m})/F$ has Galois group $H$. By the previous paragraph, the extension $F(\sqrt[n]{d^m})/F$ does not occur as a specialization of $E/F(T)$ (as $n'' \geq 2$). Hence $E/F(T)$ is not $H$-parametric, as needed.

\subsubsection{Proof of Proposition \ref{para}}

Set $${P}_0(T) = P_1(T)^{e'_1} \cdots P_s(T)^{e'_s}.$$
Then $\sqrt[n]{P(T)}$ is a $n'$-th root $\sqrt[n']{{P}_0(T)}$ of $P_0(T)$ and one has $E=F(T)(\sqrt[n']{{P}_0(T)}).$ By condition (hyp-1), $n'$ divides ${\rm{deg}}(P_0(T))$ and, by condition (hyp-2), one has $e'_j \leq n'-1$ for each $j \in \{1,\dots,s\}$. Moreover, one has ${\rm{gcd}}(n',e'_1, \dots,e'_s)=1$. As the remaining conditions (hyp-3) and (hyp-4) also hold, we may apply Proposition \ref{para 2} to the extension $F(T)(\sqrt[n']{{P}_0(T)})/F(T)$, thus ending the proof of Proposition \ref{para}.

\subsection{A quantitative version in the case $n=2$ and $F=\Qq$}

\subsubsection{Statement of Proposition \ref{para 3}}

Let $T$ be an indeterminate. Given a $\Qq$-regular quadratic extension $E/\Qq(T)$, there exists a unique separable polynomial $P_E(T) \in \Zz[T]$ with squarefree content and such that $E=\Qq(T)(\sqrt{P_E(T)})$. Define the {\it{height of $E/\Qq(T)$}} as the height of the polynomial $P_E(T)$.

Given an even positive integer $N$ and a real number $H \geq 1$, let $$T(N,H)$$ be the set of all $\Qq$-regular quadratic extensions of $\Qq(T)$ with $N$ branch points and height at most $H$ \footnote{By the Riemann-Hurwitz formula, the number of branch points of a $\Qq$-regular quadratic extension of $\Qq(T)$ is necessarily even.}. Let $$T'(N,H)$$ be the subset of all elements in $T(N,H)$ such that there exists a set of prime numbers that has positive density and that satisfies condition ($***$) of Proposition \ref{para}.

\begin{proposition} \label{para 3}
One has
$$\frac{|T'(N,H)|}{|T(N,H)|} = 1 - O \bigg(\frac{\log(H)}{\sqrt{H}} \bigg), \quad H \to \infty.$$
\end{proposition}

\subsubsection{Proof of Proposition \ref{para 3}} 

First, we need the following two number theoretic lemmas. In the statements below, we denote the Riemann $\zeta$-function by $\zeta$ and fix an integer $n \geq 2$. Given a positive real number $H$, we set $\llbracket -H, H \rrbracket = \Zz \cap [-H, H]$.

\begin{lemma} \label{lemma nt 1}
One has
\begin{align*}
~&~|\{(a_0, \dots, a_n) \in \llbracket -H, H \rrbracket^{n+1} : a_n \not=0, \, {\rm{gcd}}(a_0,\dots,a_n)=1\}|  \\
&= \frac{2^{n+1}}{\zeta(n+1)} H^{n+1} + O(H^n), \quad H \to \infty.
\end{align*}
\end{lemma}

\begin{proof}
The proof is almost identical to the proof of the lemma of \cite[\S1]{Nym72} and is then left to the reader.
\end{proof}

\begin{lemma} \label{lemma nt 2}
There exists a positive constant $C_n$ such that 
$$|\{(a_0, \dots, a_n) \in \llbracket -H, H \rrbracket^{n+1} : a_n \not=0, \, {\rm{gcd}}(a_0,\dots,a_n) \, \, is \, \, squarefree\}| $$
is asymptotic to $$C_n \cdot H^{n+1}$$ as $H$ tends to $\infty$.
\end{lemma}

\begin{proof}
Given a positive real number $H$, set $$Z(n,H)=|\{(a_0, \dots, a_n) \in \llbracket -H, H \rrbracket^{n+1} : a_n \not=0, \, {\rm{gcd}}(a_0,\dots,a_n)=1\}|$$ and, for short, denote 
$$|\{(a_0, \dots, a_n) \in \llbracket -H, H \rrbracket^{n+1} : a_n \not=0, \, {\rm{gcd}}(a_0,\dots,a_n) \, \, is \, \, squarefree\}| $$
by $f(n,H).$ One then has
\begin{equation} \label{-5}
f(n,H)=\sum_{\substack{k=1 \\ k \, \, squarefree}}^H Z \bigg(n,\frac{H}{k} \bigg)= A(n,H)+B(n,H)
\end{equation}
with 
\begin{equation} \label{-4}
A(n,H)=\sum_{\substack{1 \leq k \leq \sqrt{H} \\ k \, \, squarefree}} Z \bigg(n,\frac{H}{k} \bigg)
\end{equation}
and
\begin{equation} \label{-3}
B(n,H)=\sum_{\substack{\sqrt{H} < k \leq {H} \\ k \, \, squarefree}} Z \bigg(n,\frac{H}{k} \bigg).
\end{equation}

First, we estimate the number $A(n,H)$. By \eqref{-4} and Lemma \ref{lemma nt 1}, one has
\begin{align*}
A(n,H) &= \sum_{\substack{1 \leq k \leq \sqrt{H} \\ k \, \, squarefree}} \bigg( \frac{2^{n+1}}{\zeta(n+1)} \frac{H^{n+1}}{k^{n+1}} + O \bigg(\frac{H^n}{k^n}\bigg) \bigg) \\
&=\bigg(\frac{2^{n+1}}{\zeta(n+1)} \sum_{\substack{1 \leq k \leq \sqrt{H} \\ k \, \, squarefree}} \frac{1}{k^{n+1}} \bigg) \cdot H^{n+1} + O(H^{n+(1/2)}).
\end{align*}
We then get
\begin{equation} \label{-1.5}
A(n,H) \sim C_n \cdot H^{n+1}, \quad H \to \infty.
\end{equation}
with
$$C_n= \frac{2^{n+1}}{\zeta(n+1)} \sum_{\substack{1 \leq k < \infty \\ k \, \, squarefree}} \frac{1}{k^{n+1}}>0.$$

Now, we estimate the quantity $B(n,H)$. By \eqref{-3}, one has
\begin{align*}
B(n,H) &\leq \sum_{\substack{\sqrt{H} < k \leq {H} \\ k \, \, squarefree}} \bigg(2 \cdot \frac{H}{k} + 1 \bigg)^n \cdot 2 \cdot \frac{H}{k} \\
& \leq H \cdot (2 \sqrt{H} + 1)^n \cdot (2 \sqrt{H}).
\end{align*}
Hence 
\begin{equation} \label{-1}
B(n,H)=O(H^{(n/2)+(3/2)}), \quad H \to \infty.
\end{equation}

As $n \geq 2$, one has $(n/2)+(3/2) < n+1$. Combining \eqref{-5}, \eqref{-1.5} and \eqref{-1} then provides
$$f(n,H) \sim C_n \cdot H^{n+1}, \quad H \to \infty,$$
as needed for the lemma.
\end{proof}

Now, we estimate $|T(N,H)|$ as $H$ tends to $\infty$. Given a positive real number $H$, the cardinality $|T(N,H)|$ is equal to the number of separable polynomials $P(T) \in \Zz[T]$ that have height at most $H$, squarefree content and such that the extension $\Qq(T)(\sqrt{P(T)})/\Qq(T)$ has $N$ branch points. By (the proof of) Lemma \ref{lemma 1}, the fact that $P(T)$ is separable and as the number of branch points of a $\Qq$-regular quadratic extension of $\Qq(T)$ is even, $\Qq(T)(\sqrt{P(T)})/\Qq(T)$ has $N$ branch points if and only if $P(T)$ has degree $N$ or $N-1$ (as $N$ is even). Then 
\begin{equation} \label{1.1}
|T(N,H)|= C(N,H) + D(N,H)
\end{equation}
where 

\noindent
- $C(N,H)$ is the number of polynomials $P(T) \in \Zz[T]$ that are separable, that have degree $N$, squarefree content and height at most $H$,

\noindent
- $D(N,H)$ is the number of separable polynomials $P(T) \in \Zz[T]$ that have degree $N-1$, squarefree content and height at most $H$.

By Lemma \ref{lemma nt 2}, one has
\begin{equation} \label{1.2}
C(N,H) \sim C_N \cdot H^{N+1}, \quad H \to \infty,
\end{equation}
for some positive constant $C_N$ and, clearly, one has
\begin{equation} \label{1.3}
D(N,H) = O(H^N), \quad H \to \infty.
\end{equation}
It then remains to combine \eqref{1.1}, \eqref{1.2} and \eqref{1.3} to get
\begin{equation} \label{eq 1}
|T(N,H)| \sim C_N \cdot H^{N+1}, \quad H \to \infty.
\end{equation}

Next, we estimate $|T'(N,H)|$ as $H$ tends to $\infty$. Let $P(T) \in \Zz[T]$ be a separable polynomial. By Proposition \ref{para}, there exists a set of prime numbers that has positive density and that satisfies condition ($***$) (with respect to the quadratic extension $\Qq(T)(\sqrt{P(T)})/\Qq(T)$) if the degree of the polynomial $P(T)$ is equal to $N$ and the Galois group of $P(T)$ over $\Qq$ is isomorphic to $S_N$ (when acting on the roots of $P(T)$). Then, given a real number $H \geq 1$, one has 
\begin{equation} \label{2}
|T'(N,H)| \geq C(N,H) - E(N,H)
\end{equation}
where $C(N,H)$ is defined after \eqref{1.1} and $E(N,H)$ is the number of separable polynomials $P(T) \in \Zz[T]$ with degree $N$, height at most $H$, squarefree content and whose Galois group is isomorphic to a proper subgroup of $S_N$. By \cite[Theorem 2.1]{Coh81b}, one has
\begin{equation} \label{4}
E(N,H)= O(H^{N+(1/2)} \cdot \log(H)), \quad H \to \infty.
\end{equation}

Finally, by \eqref{2}, one has
\begin{equation} \label{7}
1 - \frac{|T'(N,H)|}{|T(N,H)|} \leq 1- \frac{C(N,H)}{|T(N,H)|} + \frac{E(N,H)}{|T(N,H)|}.
\end{equation}
for each real number $H \geq 1$. Combining \eqref{1.1}, \eqref{1.3} and \eqref{eq 1} provides
\begin{equation} \label{5}
1- \frac{C(N,H)}{|T(N,H)|} = O \bigg(\frac{1}{{H}} \bigg), \quad H \to \infty,
\end{equation}
and, by \eqref{eq 1} and \eqref{4}, one has
\begin{equation} \label{6}
\frac{E(N,H)}{|T(N,H)|} = O \bigg(\frac{\log(H)}{\sqrt{H}} \bigg), \quad H \to \infty.
\end{equation}
It then remains to combine \eqref{7}, \eqref{5} and \eqref{6} to get
$$1 - \frac{|T(N,H)|}{|T'(N,H)|} =  O \bigg(\frac{\log(H)}{\sqrt{H}} \bigg)$$
as $H$ tends to $\infty$, 
thus ending the proof of Proposition \ref{para 3}.

\bibliography{Biblio2}

\begin{thebibliography}{Nym72}

\bibitem[Ade01]{Ade01}
Clemens Adelmann.
\newblock {\em The {D}ecomposition of {P}rimes in {T}orsion {P}oint {F}ields}.
\newblock Lecture Notes in Mathematics, 1761. Springer-Verlag, Berlin, 2001.
\newblock vi+142 pp.

\bibitem[Bha13]{Bha13}
Manjul Bhargava.
\newblock Most hyperelliptic curves over $\mathbb{Q}$ have no rational points.
\newblock {\em Manuscript}, 2013.
\newblock arXiv 1308.0395.

\bibitem[Coh81]{Coh81b}
S.~D. Cohen.
\newblock The distribution of {G}alois groups and {H}ilbert's irreducibilty
  theorem.
\newblock {\em Proc. London Math. Soc. (3)}, 43(2):227--250, 1981.

\bibitem[Gra07]{Gra07}
Andrew Granville.
\newblock Rational and integral points on quadratic twists of a given
  hyperelliptic curve.
\newblock {\em Int. Math. Res. Not. IMRN}, no. 8, 2007.
\newblock Art. {I}{D} 027, 24 pp.

\bibitem[Lan02]{Lan02}
Serge Lang.
\newblock {\em Algebra}, volume 211 of {\em Graduate {T}exts in {M}athematics}.
\newblock Springer-Verlag, New York, revised third edition, 2002.

\bibitem[Leg13]{Leg13a}
Fran\c{c}ois Legrand.
\newblock Specialization results and ramification conditions.
\newblock 2013.
\newblock To appear in Israel Journal of Mathematics. arXiv 1310.2189.

\bibitem[Leg15]{Leg15}
Fran\c{c}ois Legrand.
\newblock Parametric {G}alois extensions.
\newblock {\em Journal of Algebra}, 422:187--222, 2015.

\bibitem[Leg16]{Leg16a}
Fran\c{c}ois Legrand.
\newblock On parametric extensions over number fields.
\newblock {\em {M}anuscript}, 2016.
\newblock arXiv 1602.06706.

\bibitem[MM99]{MM99}
Gunter Malle and B.~Heinrich Matzat.
\newblock {\em Inverse Galois Theory}.
\newblock Springer Monographs in Mathematics. Springer-Verlag, Berlin, 1999.

\bibitem[Nym72]{Nym72}
J.~E. Nymann.
\newblock On the probability that $k$ positive integers are relatively prime.
\newblock {\em J. Number Theory}, 4:469--473, 1972.

\bibitem[Sad14]{Sad14}
Mohammad Sadek.
\newblock On quadratic twists of hyperelliptic curves.
\newblock {\em Rocky Mountain J. Math.}, 44(3):1015--1026, 2014.

\end{thebibliography}
\bibliographystyle{alpha}

\end{document}